\documentclass[sigconf]{acmart}
\usepackage{cleveref}
\usepackage{algorithm2e}
\usepackage{multirow}
\usepackage{mathtools}
\usepackage{amsmath}
\crefname{subsection}{subsection}{subsections}
\newtheorem{theorem}{Theorem}[section]
\newtheorem{lemma}[theorem]{Lemma}
\newtheorem{corollary}[theorem]{Corollary}

\def\namedlabel#1#2{\begingroup
    #2%
    \def\@currentlabel{#2}%
    \phantomsection\label{#1}\endgroup
}
\theoremstyle{definition}
\newtheorem{definition}[theorem]{Definition}
\theoremstyle{remark}
\newtheorem{remark}[theorem]{Remark}

\makeatletter
\newcommand\footnoteref[1]{\protected@xdef\@thefnmark{\ref{#1}}\@footnotemark}
\makeatother

\DeclareMathOperator{\LCLM}{LCLM}
\DeclareMathOperator{\GCRD}{GCRD}

\DeclareMathOperator{\ord}{ord}
\DeclareMathOperator{\Hom}{Hom}
\DeclareMathOperator{\End}{End}

\DeclareMathOperator{\Minop}{MinOp}

\DeclareMathOperator{\Orders}{Orders}
\DeclareMathOperator{\AbsOrders}
{AbsOrders}

\DeclareMathOperator{\RightFactors}{RightFactors}
\newcommand{\pda}{\mathord{\downarrow}}

\newcommand{\cs}{\mathbin{\circledS}}

\definecolor{bluesequence}{HTML}{3081D0}
\definecolor{c1}{RGB}{203, 75, 48}

\copyrightyear{2024} 
\acmYear{2024} 
\setcopyright{acmlicensed}
\acmConference[ISSAC '24]{International Symposium on Symbolic and Algebraic Computation}{July 16--19, 2024}{Raleigh, NC, USA}
\acmBooktitle{International Symposium on Symbolic and Algebraic Computation (ISSAC '24), July 16--19, 2024, Raleigh, NC, USA}
\acmDOI{10.1145/3666000.3669719}
\acmISBN{979-8-4007-0696-7/24/07}

\begin{document}

\title{Solving Third Order Linear Difference Equations in Terms of Second Order Equations}

\author{Heba Bou KaedBey}
\affiliation{%
  \institution{Florida State University}
  \streetaddress{}
  \city{Tallahassee}
  \state{FL}
  \postcode{32306}
  \country{USA}}
\email{hb20@fsu.edu}

\author{Mark van Hoeij}
\affiliation{%
  \institution{Florida State University}
  \streetaddress{}
  \city{Tallahassee}
  \state{FL}
  \country{USA}
  \postcode{32306}}
\email{hoeij@math.fsu.edu}

\author{Man Cheung Tsui}
\affiliation{%
  \institution{Florida State University}
  \streetaddress{}
  \city{Tallahassee}
  \state{FL}
  \country{USA}}
\email{manctsui@gmail.com}

\renewcommand{\shortauthors}{Bou KaedBey et al.}

\begin{abstract}
    We present two algorithms for computing what we call the absolute factorization of a difference operator. We also give an algorithm for solving third order difference equations in terms of second order equations, together with applications to OEIS sequences. The latter algorithm is similar to existing algorithms for differential equations in \cite{singer1985solving, van2007solving}, except that there is an additional order one symmetric product.

\end{abstract}



\received{}
\received[revised]{}
\received[accepted]{}

\maketitle

\section{Introduction}
\label{section:introduction}
Let $D = \mathbb{C}(x)[\tau]$ be the noncommutative polynomial ring in $\tau$ over $\mathbb{C}(x)$ with multiplication $\tau\circ f(x)=f(x+1)\circ \tau$.
An element
\[
    L=a_n\tau^n+a_{n-1}\tau^{n-1}+\dots+a_{1}\tau+a_{0},\quad a_{n},a_{0}\neq 0,
\]
of $D$ is called a \emph{difference operator} of order $n$.
Such an operator acts on $f(x)\in \mathbb{C}(x)$ by 
$L(f)(x)=\sum_{i} a_i\,f(x+i)$.

A difference operator is said to be \emph{2-solvable} if it has a nonzero \emph{2-expressible} solution, i.e., a solution that can be expressed\footnote{\label{note1}using difference ring operations, indefinite summation, and interlacing; see \cite{Tsui2024eulerian} for more details} in terms of solutions of second order equations. 
The notion of a
$2$-solvable operator is a generalization of {operators whose solutions are \emph{Liouvillian} sequences, defined in \cite{hendricks1999solving}.}

Liouvillian solutions are expressed$\mbox{}^{\ref{note1}}$ in terms of solutions of first order equations. The key distinction is that a product of solutions of first order equations still satisfies a first order equation, but for order~2 this no longer holds.
Already for third order this distinction adds a case to the classification, namely case ($S^2$) below.

\begin{theorem}[{\cite[Theorem 7.1]{Tsui2024eulerian}} restated]
Let $L\in D$ be a third order linear difference operator that is $2$-solvable. Then one of the following holds:
\begin{enumerate}
    \item [\namedlabel{case-1}{(R)}] (Reducible case). $L$ admits a nontrivial factorization over $\mathbb{C}(x)$;
    \item[(L)] \label{case-2} (Liouvillian case). $L$ is gauge equivalent to $\tau^{3}+a$ for some $a\in \mathbb{C}(x)$; (See \Cref{def:gauge} for a definition of gauge equivalence.)
    \item[($S^2$)] \label{case-3} (Symmetric square case). $L$ is gauge equivalent to $L_{2}^{\cs 2}\cs L_{1}$ for some first order operator $L_{1}$ and some second order operator $L_{2}$ over $\mathbb{C}(x)$. (See \Cref{def:symm-prod} on symmetric product.)
\end{enumerate}
\end{theorem}

Cases (R) and (L) are already treated by existing algorithms in \cite{petkovvsek1992hypergeometric, barkatou2024hypergeometric, hendricks1999solving}. 
One main goal of this paper is to give an algorithm in \Cref{subsection:reduce-order} to deal with case ($S^2$).

This paper is organized as follows. In Section 2, we review definitions and results on difference operators and modules. In Section 3, we give an algorithm for \emph{absolute factorization}. As illustrated in Example~\ref{A260772}, this covers one of the cases in the classification (future work, \Cref{futurework}) of $2$-solvable order 4 equations. Moreover, we will describe the orders of the resulting factors in this factorization through the notion of \emph{absolute order} in \Cref{subsection:absolute-order}.

Section 4 gives a second algorithm for absolute factorization. Section 5 covers case ($S^2$): given $L \in D$ of order~$3$, decide if case ($S^2$) holds, and if so, return $L_2$, $\tau - r$, and the gauge transformation.
We give an example from the OEIS where our algorithm produces an output that proves a conjecture from Z.-W. Sun.

The first two authors were supported by NSF grant 2007959.

\section{Preliminaries}

This section introduces some notations and facts about difference operators and modules that we will use. The standard facts we use are from \cite{hendricks1999solving, van2006galois}.

Consider the ring $S=\mathbb{C}^{\mathbb{N}}/\sim$ 
under the equivalence relation $u \sim v$ if $(u-v)(x)\neq 0$ 
for finitely many $x$. An element of $S$ is called a \emph{sequence} and represented as a function $u:\mathbb{N}\to \mathbb{C}$ or a list $(u(0),u(1),\dots)$. 
The set $S$ is a difference ring under component-wise addition and multiplication, and $\tau(u(x))=u(x+1)$.
Since rational functions have finitely many poles,
$\mathbb{C}(x)$ embeds into $S$ by evaluation at $\mathbb{N}$, making $S$ a (left) $D$-module where $D=\mathbb{C}(x)[\tau]$.

The \emph{solution space} $V(L)$ of an operator $L$ is the set $\{u\in S\mid Lu=0\}$. 
This is a $\mathbb{C}$-vector space of dimension $\ord(L)$ by \cite[Theorem 8.2.1]{petkovvsek1997wilf}.

We can define some notions about operators in terms of $D$-modules.
If $u$ is an element of a $D$-module, the \emph{minimal operator} of $u$, denoted by $\Minop(u,D)$, 
is the monic generator of the left ideal $\{L\in D \mid L(u)=0\}$ of $D$. 
Given $L\in D$ and $u\in V(L)$, the Maple command \texttt{MinimalRecurrence} computes $\Minop(u,D)$.

Let $M$ be a $D$-module. Then $u\in M$ is a \emph{cyclic vector} if $Du=M$.
We call $M$ \emph{irreducible} if every nonzero element of $M$ is cyclic; 
equivalently, the minimal operator of each element of $M$ is irreducible of order $\dim(M)$.

For $L_{1},L_{2}\in D$, the operators $\GCRD(L_1,L_2)$ and $\LCLM(L_1,L_2)$ are defined as the monic generators of the left ideals $DL_1+DL_2$ and $DL_1\cap DL_2$ of $D$, respectively. Then the following facts hold.
\begin{enumerate} 
    \item $L_1$ is a right factor of $L_2$ if and only if $V (L_1)\subseteq V(L_2)$. 
    \item $V(\GCRD(L_1,L_2))=V(L_1)\cap V(L_2)$.
    \item $V(\LCLM(L_1,L_2))=V(L_1)+V(L_2)$.
\end{enumerate} 

Here are some more constructions on difference operators and modules we will use.

\begin{definition}
\label{def:symm-prod}
    Let $M$ and $N$ be $D$-modules.
    The \emph{tensor product} $M\otimes N$ is a $D$-module under $\tau(m \otimes n)=\tau(m) \otimes \tau(n)$. 
    The \emph{symmetric product} $L_1 \cs L_2$ of $L_1,L_2 \in D$ is the minimal operator of $1\otimes 1$ in $(D/DL_{1})\otimes(D/DL_{2})$. 
    The \emph{symmetric square} of $L$ is $L^{\cs 2}\coloneqq L\cs L$.
\end{definition}

\begin{lemma} [{\cite[Proposition 2.13]{van2012galois} or \cite[Lemma 2.5]{singer1996testing}}] 
\label{lemma:homomorphism} 
    Let $L_1,L_{2}\in D$. Then the map 
    \begin{align*}
        \Psi: \Hom_{D}(D/DL_1,D/DL_2) &\longrightarrow \Hom_{\mathbb{C}}(V(L_2),V(L_1)) \\
        \phi &\mapsto G_{\phi}:V(L_2) \to V(L_1)& 
    \end{align*}
    is injective, where $G_{\phi}$ is defined by $\phi(\bar 1)=G_{\phi}+DL_2$.
\end{lemma}

Homomorphisms between $D$-modules can be computed by \cite{barkatou2024hypergeometric,Imp} and help detect reducibility of operators:

\begin{lemma}[{\cite[Corollary 2.7]{singer1996testing}}]
\label{lemma:reducible-operator} 
    If $\End_{D}(D/DL) \neq \mathbb{C}$, then $L$ is reducible.
\end{lemma}

\begin{lemma} 
\label{lemma:gauge}
    Let $L_1,L_2\in D$ have the same order. The following statements are equivalent:
    \begin{enumerate} 
        \item $D/DL_1\cong D/DL_2$ as $D$-modules.
        \item There exists $G \in D$ such that $G(V(L_2))=V(L_1).$
        \item Let $G$ be as in (2), and such that $L_1G \in DL_2$ and $\GCRD(G,L_2)=1.$
        \item There exists a $D$-module with cyclic vectors $u_1,u_2$ such that $L_1=Minop(u_1,D)$ and $L_2=Minop(u_2,D).$  
    \end{enumerate}  
\end{lemma}

\begin{definition}  \label{def:gauge} If any statement in Lemma \ref{lemma:gauge} holds, we say that $L_1$ is \emph{gauge equivalent} to $L_2$,
    and $G$ is a \emph{gauge transformation} from $L_{2}$ to $L_{1}$.
    Moreover,
    $G: V(L_2) \rightarrow V(L_1)$ is a bijection.  \end{definition}
    
\begin{remark} \label{remark:gauge} Computing
the inverse of $G$. Since $\GCRD(G,L_2)=1$ in $D$, the extended Euclidean algorithm finds $\widetilde{G},R \in D$ with $\widetilde{G}G+RL_2=1$. Thus $\widetilde{G}G\equiv 1\pmod{L_{2}}$ and $\widetilde{G} G$ is the identity map on $V(L_2)$. Therefore $\widetilde{G}: V(L_1) \to V(L_2)$ is the inverse of $G$.\end{remark}

\section{Absolute Factorization}
In this section, we examine when, and in what sense,
an irreducible operator can be ``factored'' further, or else be
``absolutely irreducible''.
{The motivation is that this is needed when we want to extend our work to order $>3$; see Section \ref{futurework}.}

\begin{definition}[{\cite[Definition 5.1]{Tsui2024eulerian}}]
    Let $D_{m}=\mathbb{C}(x)[\tau^m] \subseteq D$, where $m \in \mathbb{Z}^{+}$.
    Let $M$ be a $D$-module.
    Then $M\pda^{1}_{m}$ is $M$ viewed as a $D_{m}$-module.
\end{definition}

A submodule of $M\pda^{1}_{m}$ is simply a subset of $M$ that is also a $D_m$-module. Since every $D$-module is a $D_m$-module, but not vice versa,
$M\pda^{1}_{m}$ could have more submodules than $M$.

The following is an isomorphism between $D$ and $D_m$:
\begin{equation}\begin{split} \label{isom-D-Dm}
 \psi_{m} :D &\longrightarrow D_m \\
  \tau &\mapsto \tau^m \\
  x &\mapsto \frac{x}{m}\end{split}\end{equation} 

\begin{definition}
\label{def:section-operator}
    Let $L\pda^{1}_{m}$ be the minimal operator of $\bar{1} \in (D/DL)\pda^1_m$, 
    i.e., the monic element of minimal order in $D_{m}\cap DL$. The $m$th \emph{section operator} is defined as $L^{(m)}\coloneqq \psi_{m}^{-1}(L\pda^{1}_{m}) \in D$.
\end{definition} 

Note that $L\pda^{1}_{m}$ is the same as $P(\phi^m)$ in \cite[Lemma 5.3]{hendricks1999solving} while $L^{(m)}$ is $P_{0}(\phi)$ in \cite[Lemma 5.3]{hendricks1999solving}. If $\overline{1}$ generates a proper $D_m$-submodule of $(D/DL)\pda^1_m$ then we say that the order of $L^{(m)}$ is \emph{lower than expected} (lower than the order of $L$).

\begin{definition} A $D$-module $M$ is \emph{absolutely irreducible} if $M\pda^{1}_{t}$ is irreducible for every $t \ge 1$. An operator $L$ is \emph{absolutely irreducible} if $M = D/DL$ is \emph{absolutely irreducible}.\end{definition}

For difference and differential operators, $L$
is absolutely irreducible when its solution space is 
an irreducible $G^o$-module, where $G^o$ is the connected component
of the identity of the Galois group. In the differential case
this occurs when $L$ is irreducible over algebraic extensions
of $\mathbb{C}(x)$.
In the difference case, such extensions are replaced with section operators.

\begin{definition} 
\label{def:ord-m}
    Let $m$ be in $\mathbb{Z}^{+}$ and let $L$ be in $D_m$. 
    We define $\ord_{m}(L)\coloneqq \ord(L)/m$, the highest power of $\tau^m$ that appears in $L$.
\end{definition}

\begin{theorem} 
\label{thm:orders}
Let $L$ be irreducible in $D$ and let $M = D/DL$. Let $m$ be in $\mathbb{Z}^{+}$. Then the following are equivalent.
\begin{enumerate} 
    \item $M\pda^{1}_{m}$ is a reducible $D_m$-module.
    \item $L\pda^{1}_{m}$ is reducible in $D_m$ or $\ord_{m}(L\pda^{1}_{m})<\ord(L)$.
    \item $L^{(m)}$ is reducible in $D$ or  $\ord(L^{(m)})<\ord(L)$.
    \item $M\pda^{1}_{p}$ is reducible for some prime $p$ dividing $\gcd(m, \dim(M))$.
\end{enumerate} 
\end{theorem}

\begin{proof} \hfill

(1) $\Leftrightarrow$ (2): 
If $\bar{1} = 1 + DL$ generates $M\pda^{1}_{m}$, then its minimal operator $L\pda^{1}_{m}$ has the expected order, and is reducible in $D_m$ if and only if $M\pda^{1}_{m}$ is reducible. If $\bar{1}$ does not generate $M\pda^{1}_{m}$, then it generates a nontrivial submodule, and $L\pda^{1}_{m}$ has lower order than expected.

(2) $\Leftrightarrow$ (3): 
Apply the isomorphism $\psi_{m}:D\to D_m$ (Equation (\ref{isom-D-Dm})).

(4) $\Rightarrow$ (1): 
Clear because $D_m \subseteq D_p$.

(1) $\Rightarrow$ (4): 
This is precisely \cite[Corollary 6.6]{Tsui2024eulerian}.
\end{proof}

As mentioned in the above proof, if $M=D/DL$ then $\bar{1}=1+DL$ generates $M$ as a $D$-module, but it might not generate $M$ as $D_m$-module, in other words, $D_m\cdot \bar{1} \subsetneq M \pda^{1}_{m}$ is a proper submodule. 
That case corresponds to Step~1b in the algorithm below. 
The other option is that $\bar{1}$ does generate $M \pda^{1}_{m}$. 
In that case, 
having a proper submodule is equivalent to $L\pda^1_m$ having a proper factor in $D_m$, and $L^{(m)}$ having a proper factor in $D$, 
which corresponds to Step~1c in the algorithm below.

\subsection{Absolute Factorization Algorithm} \label{abs}

\Cref{thm:orders}~(4) immediately gives the correctness of the following algorithm. The implementation is in \cite{algo}.

\textbf{Algorithm:} \texttt{AbsFactorization}\\
\textbf{Input:} An irreducible operator $L\in D$.\footnote{\label{foot:note1} Factoring is explained in \cite{barkatou2024hypergeometric,zhou2022algorithms}. The Maple command \texttt{RightFactors}, implemented by van Hoeij, factors $L$ if it is reducible. By default, it factors in $C(x)[\tau]$ where $C$ is the smallest field of constants over which the input is defined. To ensure a full factorization over $\mathbb{C}$ one may need to add field extensions to an optional input of \texttt{RightFactors} to ensure that $C$ is large enough.}\\
\textbf{Output:} \texttt{Absolutely\,Irreducible} if $L$ is absolutely irreducible; otherwise $[p,\{R_i\}]$ where $p$ is a prime and where each $R_i$ right-divides $L^{(p)}$ {and $\psi_{p}(R_i)$ generates a nontrivial submodule of $M \pda^{1}_{p}$, with $M = D/DL$.}

\begin{enumerate}
    \item For each prime factor $p$ of $\ord(L)$\,: \begin{enumerate}
        \item Compute $L^{(p)}$ from Definition \ref{def:section-operator}.
        \item If $\ord(L^{(p)}) < \ord(L)$, then return $[p,\{1\}]$. \label{1b}
        \item Compute $S\coloneqq \RightFactors\Big(L^{(p)},\frac{\ord(L)}{p}\Big)$. \label{1c}
        \item \label{1d} If $S \neq \varnothing$, then return $[p,S]$ and stop.
    \end{enumerate}
    \item Return \texttt{Absolutely\,Irreducible}.
\end{enumerate}

\subsection{Example: OEIS A260772} \label{A260772}

In the OEIS database,
the sequence $A260772$
\begin{equation} 
\label{sequence}
  \big(\textcolor{bluesequence}{1}, 3, \textcolor{bluesequence}{10}, 41, \textcolor{bluesequence}{190}, 946, \ldots \big) \in \mathbb{C^{\mathbb{N}}}
\end{equation}
counts certain directed lattice paths. 
It has the minimal operator
\begin{equation*}
    \begin{split}
        L \coloneqq & (x+5)(x+4)(25x^2+130x+141)\tau^4-30(x+4)(7x+13)\tau^3 \\
        & - (1100x^4+12320x^3+48664x^2+80740x+47400)\tau^2  \\
        & + 120(x+6)(x+1)\tau-16x(x+1)(25x^2+180x+296).
    \end{split}
\end{equation*}
Even though $L$ is irreducible of order 4, it is nevertheless 2-solvable, as we now show.
The subsequence of \eqref{sequence},
\[u(n)= A260772(2n)= (\textcolor{bluesequence}{1, 10, 190, \ldots}) \in \mathbb{C}^{\mathbb{N}} \]
is called a \emph{2-section} of \eqref{sequence}.  
The other 2-section is $A260772(2n+1) = 3, 41, 946, \ldots$

The minimal recurrence of $u(n)$ is the \emph{2-section operator} from \Cref{def:section-operator}:
\begin{equation*} \textcolor{bluesequence}{
    \begin{split}
        L^{(2)} & =   (4x^4 + 56x^3 + 287x^2 + 634x + 504)\tau^4 \\
         & +(-352x^4 - 4048x^3 - 17276x^2 - 32354x - 22344)\tau^3 \\
         & +(7616x^4 + 68544x^3 + 229648x^2 + 339408x + 186648)\tau^2  \\
        & +(5632x^4 + 36608x^3 + 86336x^2 + 88288x + 32928)\tau  \\
        & +1024x^4 + 4096x^3 + 4352x^2 + 1280x. 
    \end{split}}
\end{equation*}

But this operator is reducible! 
Namely, 
\begin{equation} 
\label{eq3} 
    L^{(2)}=\LCLM(R,R') 
\end{equation} 
where $R$ and $R'$ are irreducible second order operators given in \eqref{eq:factors} below. 
This means that the 2-section $u(n)$ of A260772 is {2-expressible}. 
Equation \eqref{eq3} expresses the fact that $u(n)$ can be written as 
\begin{equation}           
\label{solution}
    u(n)=t(n)+t'(n)
\end{equation} 
for some $t$, $t' \in \mathbb{C}^{\mathbb{N}}$ that satisfy $R$, $R'$. Computing bases of solutions of $R$ and $R'$ in $\mathbb{C}^{\mathbb{N}}$ and comparing initial conditions gives $t(0)=0,t(1)=4$, and $t'(0) = 1, t'(1)=6$. 
The recurrence $R(t)=0$ reads:
\begin{equation} 
    \label{recurrence} t(x+2) =  \frac{a_0t(x)+a_1t(x+1)}{(x+2)(2x+5)(5x+3)}
\end{equation}
which expresses $t(2), t(3),\ldots$ in terms of prior $t$ values. 
Here $a_0 = 16x(2x+1)(5x+8)$ and $a_1 = 440x^3+1584x^2+1780x+600$.

We can express $t'(n), n \in \mathbb{N}$ by using~\eqref{recurrence} for $x \in -\frac{1}{2}+\mathbb{N}$, as follows. Let $t(-\frac{1}{2}) =  -\frac{1}{2}$ and $t(\frac{1}{2}) = 1$, 
and define $t(x+2)$ for $x \in -\frac{1}{2}+\mathbb{N}$ with \eqref{recurrence}.
By computing a gauge transformation from $V(R\vert_{x \mapsto x + \frac12})$ to $V(R')$, we find
\begin{equation} \label{un}
    t'(n) = \frac{(2-4n)t\big(n-\frac{1}{2}\big) + (2n+2)t\big(n+\frac{1}{2}\big)}{1+10n}.
\end{equation}
See \cite{algo} for further details.

Recall that $u(n) = A260772(2n)$.
Substituting~\eqref{un} and $n \mapsto n/2$
in~\eqref{solution} gives a formula (that we uploaded to OEIS) for $A260772(n)$, 
written in terms of solutions (in $\mathbb{C}^{\mathbb{N}}$ and in ${\mathbb{C}}^{-\frac{1}{2} + \mathbb{N}}$)
of a \emph{second order} 
recurrence~\eqref{recurrence}.

This example illustrates one of the cases in the classification (work in progress, Section \ref{futurework}) of $2$-solvable order $4$ equations.
Algorithm \texttt{AbsFactorization} covers this case, which we will illustrate with the same example.

$L$ has order $4$ so the only prime $p$ in Step 1 is $p=2$. 
Step~1a computes $L^{(2)}$.
It has order 4 (Step~1b) so we proceed with Step~1c which computes the set of right factors $S = \{R, R'\}$ where

\begin{equation}
\label{eq:factors}
    \begin{split}
      R &= (2x + 5)(5x + 3)(x + 2)\tau^2\\
        -&(440x^3 + 1584x^2 + 1780x + 600)\tau
        - 8(5x + 8)(4x^2 + 2x), \\
      R' &= (2x + 5)(10x + 9)(x + 2)\tau^2\\
        -& (880x^3 + 3432x^2 + 4220x + 1650)\tau
        - 16(10x + 19)(2x^2 + x).
    \end{split}
\end{equation}

$R$ and $R'$ generate nontrivial submodules of $D/D L^{(2)}$. 
{Applying $\psi_{2}$ gives generators of the corresponding submodules of $(D/DL)\pda^1_2$.}

\subsection{Absolute Order}
\label{subsection:absolute-order}
The right factors $R$ and $R'$ in \eqref{eq:factors} exemplify a general phenomenon: 
for an irreducible operator $L$, all irreducible factors of $L^{(m)}$ will have the same order. 
This is a consequence of \Cref{cor:abs-ord} below, using the concept of absolute order which we now define.

\begin{definition}
    Let $M$ be a $D$-module. Let the sequence
    \begin{equation*} 
        0=M_0 \subsetneq M_1 \subsetneq \dots \subsetneq M_{k-1} \subsetneq M_k=M
    \end{equation*}
    of $D$-submodules be a composition series of $M$, i.e., for each $i$, $M_{i+1}/M_i$ is irreducible. Then  
    \begin{enumerate}
        \item $\Orders_{D}(M)$ is defined as the sorted list of numbers\\ $\dim_{\mathbb{C}(x)}(M_{i+1}/M_{i})$. 
        \item $\AbsOrders(M)$ is defined as $\Orders_{D_t}(M\pda^{1}_{t})$ for a value of $t$ for which $M\pda^{1}_{t}$ has the maximum number of composition factors.
    \end{enumerate}
\end{definition}

Note that $\Orders_{D}(M)$ is an ordered partition of $\dim_{\mathbb{C}(x)}(M)$ and independent of the choice of composition series by the Jordan-Hölder theorem. 
Likewise, $\AbsOrders(M)$ is well-defined: two composition series for $M\pda^{1}_{s}$ and $M\pda^{1}_{t}$ maximizing the number of composition factors must restrict to two composition series of $M\pda^{1}_{st}$ by the maximality assumption (Any $D_s$-module is also a $D_{st}$-module). By the Jordan-Hölder theorem, their composition factors coincide, so $\Orders_{D_s}(M)=\Orders_{D_t}(M)$.

\begin{corollary} [of Theorem \ref{thm:restrict-to-p} below]
\label{cor:abs-ord}
If $M$ is an irreducible $D$-module of dimension $n$, then 
\begin{enumerate} 
    \item\label{item:abs-1} $\AbsOrders(M) = [m,\dots,m]$ ($k$ copies of $m$) for some $k,m$ with $km=n$. 
    Moreover, for any $t$ we have:
    \item\label{item:abs-2}  $\Orders_{D_t}(M) = [m,\dots,m]$ ($k$ copies of $m$) if and only if $k\mid t$. 
\end{enumerate} \end{corollary} 

\begin{proof} \hfill

\eqref{item:abs-1}: If $M$ is absolutely irreducible, then $m=n$ and $k=1$.
Otherwise, $M\pda^{1}_{p}$ is reducible for some prime factor $p$ of $n$ by \Cref{thm:orders}.
By \Cref{thm:restrict-to-p}, $M\pda^{1}_{p} \cong N\oplus\tau(N)\oplus\cdots\oplus \tau^{p-1}(N)$.
After applying the isomorphism in Equation~\eqref{isom-D-Dm}, we invoke induction to conclude that $\AbsOrders(N)=[m,\dots,m]$ ($k'$ copies of $m$) for some $m,k'$ with $k' m =n/p$.
Therefore, 
\[
    \AbsOrders(N\oplus\tau(N)\oplus\cdots\oplus \tau^{p-1}(N))= [m,\dots,m]\quad (k \text{ copies of } m)
\]
with $k=pk'$.

\eqref{item:abs-2}: Argue as in \eqref{item:abs-1}.
\end{proof}

\Cref{cor:abs-ord} also follows from \cite[Proposition 6.8]{Tsui2024eulerian}.

\section{Another Algorithm for Absolute factorization}

\label{anotheralg} 

Implementation is in \cite{algo}. Let $\zeta_{p}$ be a primitive $p$th root of unity.

\textbf{Algorithm:} \texttt{AbsIrreducibility}\\
\textbf{Input:}  An irreducible operator $L\in D$.\footnoteref{foot:note1}\\
\textbf{Output:} \texttt{True} if $L$ is absolutely irreducible; otherwise \texttt{False}.
\begin{enumerate}
\item
For each prime $p\mid\ord(L)$:
    \begin{enumerate}
        \item Compute $\widetilde{L}\coloneqq L \vert_{\tau \mapsto \tau/\zeta_p}$ which equals $L \cs (\tau-\zeta_{p})$.
        \item Compute the set $\Hom_{D}(D/D\widetilde{L}, D/DL)$ using \cite{barkatou2024hypergeometric,Imp}.\\
        If this set is nontrivial, return \texttt{False}.
    \end{enumerate}
\item Return \texttt{True}.
\end{enumerate}

Correctness of this algorithm follows from

\begin{theorem} \label{thm:restrict-to-p}
    Let $L$ be irreducible in $D$. Take $M = D/DL$ and let $p$ be a prime.
    Then the following are equivalent: 
    \begin{enumerate} 
        \item\label{item:restrict-to-p:1} $L$ is gauge equivalent to an element of $D_p$;
        \item\label{item:restrict-to-p:2} $L$ is gauge equivalent to $L \cs (\tau - \zeta_p)$;
        \item\label{item:restrict-to-p:3} $M \cong M \otimes D/D(\tau - \zeta_p)$ as $D$-modules; 
        \item\label{item:restrict-to-p:4} $M\pda^{1}_{p}$ is reducible.
        \item\label{item:restrict-to-p:5} $M\pda^{1}_{p}\cong N\oplus\tau(N)\oplus\cdots\oplus\tau^{p-1}(N)$ for some irreducible $D_{p}$-submodule $N$ of $M\pda^{1}_{p}$.
    \end{enumerate} 
\end{theorem}

\begin{proof} \hfill

\eqref{item:restrict-to-p:1} $\Rightarrow$ \eqref{item:restrict-to-p:2}: 
Suppose that $L$ is gauge equivalent to $L'\in D_p$. 
Then $L$ is also gauge equivalent to $L \cs(\tau - \zeta_p)$ since $L' = L' \cs (\tau - \zeta_p)$.

\eqref{item:restrict-to-p:2} $\Rightarrow$ \eqref{item:restrict-to-p:3}: 
$L \cs (\tau - \zeta_p)$ is the minimal operator for $1 \otimes 1$ in $M \otimes D/D(\tau - \zeta_p)$.

\eqref{item:restrict-to-p:3} $\Rightarrow$ \eqref{item:restrict-to-p:4}: 
$D/D(\tau - \zeta_p)$ and $D/D(\tau-1)$ are isomorphic as $D_p$-modules but not as $D$-modules.
The composite $D_{p}$-module isomorphism 
$$
    M\cong M\otimes D/D(\tau-\zeta_{p})\cong M\otimes D/D(\tau-{1})\cong M
$$
is therefore nontrivial. So $M\pda^{1}_{p}$ is reducible by \Cref{lemma:reducible-operator}. 

\eqref{item:restrict-to-p:4} $\Rightarrow$ \eqref{item:restrict-to-p:5}:
Since the prime $p$ has no proper divisors, by \cite[Theorem 6.5]{Tsui2024eulerian}, $M\cong D\otimes_{D_p}N$ for some irreducible $D_{p}$-module $N$. 
Now use the fact that $D$ is a free left $D_{p}$-module with basis $1,\tau,\tau^{2},\dots,\tau^{p-1}$.

\eqref{item:restrict-to-p:5} $\Rightarrow$ \eqref{item:restrict-to-p:1}:
By the cyclicity of the direct sum decomposition, a cyclic vector $u$ of $N$ is also cyclic for $M$ with minimal operator in $D_{p}$.
\end{proof}

If $L$ is not absolutely irreducible, we can find a factorization
by adjusting \texttt{AbsIrreducibility} with the proof of \cite[Theorem 6.4]{Tsui2024eulerian}. We illustrate this for $p=2$. 

Suppose $\Hom_{D}(D/D\widetilde{L}, D/DL)$ from Step~1b has a nonzero element $G$. Then nontrivial factors of $L\pda^1_2$ in $D_2$ are constructed as follows: 
\begin{enumerate}
    \item Let $c$ be the remainder of $\widetilde{G} G \pmod{DL}$
    where $\widetilde{G} = G\vert_{\tau \mapsto -\tau} \ \in \Hom_{D}(D/D{L}, D/D\widetilde{L})$.
    Since $L$ is irreducible, $c$ must be a constant by \Cref{lemma:reducible-operator}. 
    (This $c$ is the same as in the proof of Theorem $6.5$ in \cite{Tsui2024eulerian}).
    \item Replace $G$ with ${G}/{\sqrt{c}}$.
    \item Now $1 + G$ resp. $1-G$ are maps from $V(L)$ to $V(L) + V(\tilde{L}) = V(L\pda^1_2)$. The images of these maps are solution spaces of right-factors of $L\pda^1_2$. Remark~\ref{remark:gauge} shows how to compute them, which, due to $\tau \mapsto -\tau$ symmetry, will be in $D_2$. This way the Hom computation in Step 1(b) replaces a call to \texttt{RightFactors}. 
\end{enumerate}

This $1+G$ is the same as $\Psi$ in the proof of \cite[Theorem 6.5]{Tsui2024eulerian}. 
For $p>2$, \texttt{AbsIrreducibility} and its extension above require the use of the field extension $\mathbb{Q}(\zeta_p)/\mathbb{Q}$. This may impact the efficiency of \texttt{AbsIrreducibility} compared to that of \texttt{AbsFactorization}.

If $p = {\rm ord}(L)$ then \texttt{AbsFactorization} becomes equivalent to computing Liouvillian solutions, for which improvements were given in~\cite{Liouv2009}. Our next goal is to implement similar improvements in \texttt{AbsFactorization}, by modifying \texttt{RightFactors}, and then compare the implementations. The results will be posted on our websites~\cite{algo}.

\section{Solving third order Equations in Terms of second order equations}

As explained in the introduction, to solve third order equations in terms of second order equations, it suffices to cover cases (R), (L) and ($S^2$). Cases (R) and (L) are handled by (absolute) factorization, so we focus on case ($S^2$). This can be broken down into two steps:
\begin{enumerate}
    \item Find, if it exists, a gauge transformation $G$ from $L_3$ to a third order operator $L_G$ for which $\ord (L_G^{\cs 2})=5$. This is done in \Cref{subsection:reduce-order}.
    \item Decompose $L_{G}$ as $L_2^{\cs 2}\cs (\tau-r)$ using case~(c) in \Cref{thm:symmetric-square-invariant}(3). (Cases (a) and (b) in \Cref{thm:symmetric-square-invariant}(3) are already covered by cases (L) and (R)).
\end{enumerate}

\subsection{Simpler Case}
\label{subsection:simpler-case}

The easiest sub-case of case ($S^2$) in Section 1 is when that gauge-equivalence is simply an equality: $L_3 = L_2^{\cs 2}\cs (\tau-r)$.
This section shows how to detect this sub-case, and if so, how to find $L_2$ and $r$.

\begin{theorem}\label{thm:symmetric-square-invariant}
 Let $L_3=\tau^3+c_2\tau^2+c_1\tau+c_0$ be a third order difference operator over $\mathbb{C}(x)$. The following are equivalent. 
\begin{enumerate}
\item \label{thm:invar1} The element $I(c_0,c_1,c_2)$ given by
    \begin{equation}
\begin{aligned}
& + \tau^2c_0 \cdot c_1 \cdot c_2\cdot (\tau c_2)^2\cdot \tau^2 c_2 
& -\tau^2c_0\cdot  c_1 \cdot \tau^2 c_1 \cdot c_2\cdot \tau c_2\\
& - \tau^2c_{0} \cdot c_{2} \cdot \tau c_{2}\cdot \tau c_{0} \cdot \tau^2c_{2} 
& -\left(\tau c_{1}\right)^2 \cdot c_{1}\cdot \tau^2 c_{1} \cdot \tau^2 c_{2}\\
& + \tau c_{0}\cdot \tau c_{1} \cdot \tau^2 c_{1} \cdot c_{2} \cdot \tau^2 c_{2} 
& +\tau^2 c_{0} \cdot \tau c_{1} \cdot c_{1} \cdot \tau^2 c_{1}
\end{aligned}
\end{equation}
is zero in $\mathbb{C}(x)$.
\item \label{thm:invar2}  The operator $L_3^{\cs 2}$ has order $\le 5$.
 \item\label{thm:invar3} One of the following three cases holds.
    \begin{enumerate}
    
\item $L_3 = \tau^3+c_0$, with $c_0 \neq 0$;

\item $L_3 = (\tau+c_2)\circ (\tau^2+\tau^{-1}(c_1))$;

\item Let $L_2 = \tau^2+a\tau+b$.
For some $a,r\in \mathbb{C}(x)^*$ and $b \neq 1$, we have
\begin{equation}\label{eq:symmetric-square-term-equivalence} L_3 = L_2^{\cs 2}\cs (\tau-r). \end{equation}

In this case we may choose
\begin{equation*}
a=1,\quad b = \dfrac{
\begin{vmatrix}
1 & \frac{\tau^{-1}(c_2) \, c_2}{c_1}\\
1 & 1
\end{vmatrix}
}{
\begin{vmatrix}
1 & \frac{\tau^{-1}(c_2) \, c_2}{c_1}\\
\frac{c_1 \, \tau^{-1}(c_1)}{c_0 \, \tau^{-1}(c_2)} & 1
\end{vmatrix}
},\quad r = \dfrac{\tau^{-2}(c_2)}{\tau^{-1}(b) - 1}.
\end{equation*}
    \end{enumerate}
    \end{enumerate}
\end{theorem}

\begin{proof}
A computer verification shows the $\tau^6$-coefficient of $L_3^{\cs 2}$ is $I(c_0,c_1,c_2)$, so \eqref{thm:invar1} and \eqref{thm:invar2} are equivalent. We now take the three special cases and study them.

Case (a). Suppose $c_1$ or $c_2$ is zero. By inspection, we see that $I$ is zero precisely when $c_1=c_2=0$, i.e., when $L_3 = \tau^3+c_0$.

Case (b). Suppose $c_0 = \tau^{-1}(c_1)c_2$. Then $I=0$ and $L_3$ factors as shown.

Case (c). Suppose that $c_1,c_2\neq 0$ and $c_0 \neq \tau^{-1}(c_1)c_2$. We consider when \eqref{eq:symmetric-square-term-equivalence} holds. Now $a \neq 0$, otherwise
the right side of \eqref{eq:symmetric-square-term-equivalence} would have order $2$. Replacing $L_2$ with
\[
    L_2 \cs (\tau - 1/a) = \tau^2+\tau +b'
\]

and multiplying $r$ by $a^2$
does not change the right side of \eqref{eq:symmetric-square-term-equivalence}, so we may further assume $a=1$. 

Matching coefficients in \eqref{eq:symmetric-square-term-equivalence} now gives a system of equations
\begin{equation}\label{eq:system}
    \begin{aligned}
    {[\tau^2]}&:& c_2 &= \tau^2(r)[\tau(b)-1], \\
    {[\tau]}&:& c_1 &= -\tau(r)\tau^2(r)\tau(b)[\tau(b)-1],\\
    {[1]}&:& c_0 &= -r \, \tau(r) \tau^2(r) \, b^2 \, \tau(b).
    \end{aligned}
\end{equation}
Since $c_1,c_2\neq 0$, \eqref{eq:system} is equivalent to
\begin{equation}
\label{eq:system2}
    \begin{aligned}
    {[\tau^2]}&:& c_2 &= \tau^2(r)[\tau(b)-1], \\
    \frac{\tau^{-1}[\tau^2]\cdot [\tau^2]}{[\tau]}&:& \frac{\tau^{-1}(c_2) \, c_2}{c_1} &= \frac{1-b}{\tau(b)}, \\
    \frac{[\tau]\cdot\tau^{-1}[\tau]}{[1]\cdot\tau^{-1}[\tau^2]}&:& \frac{c_1\,\tau^{-1}(c_1)}{c_0 \, \tau^{-1}(c_2)} &= \frac{1-\tau(b)}{b}.
    \end{aligned}    
\end{equation}
The last two equations can be written as the linear system
\begin{equation*}\label{eq:matrix-eq}
A
\begin{bmatrix}
b\\ \tau(b)
\end{bmatrix}
=
\begin{bmatrix}
1\\1
\end{bmatrix},\qquad
A=\begin{bmatrix}
1 & \dfrac{\tau^{-1}(c_2) \, c_2}{c_1}\\
\dfrac{c_1 \, \tau^{-1}(c_1)}{c_0 \, \tau^{-1}(c_2)} & 1
\end{bmatrix}.
\end{equation*}

Since $c_0 \neq \tau^{-1}(c_1)c_2$, we have $\det A\neq 0$. Cramer's rule now gives 
\begin{equation}
b = \frac{1}{\det(A)}
\begin{vmatrix}
1 & \frac{\tau^{-1}(c_2) \, c_2}{c_1}\\
1 & 1
\end{vmatrix},
\quad
\tau(b) = \frac{1}{\det(A)}
\begin{vmatrix}
1 & 1\\
\frac{c_1 \, \tau^{-1}(c_1)}{c_0 \, \tau^{-1}(c_2)} & 1
\end{vmatrix}.\label{eq:cramers}
\end{equation}

Consistency of \eqref{eq:system} is now equivalent to the condition $\tau(f) -g = 0$, where $f$ and $g$ are the right sides in \eqref{eq:cramers}. But the numerator of $\tau(\tau(f) -g)$, written as a rational function of $\tau^i(c_j)$, is precisely $I(c_0,c_1,c_2)$. Therefore, \eqref{eq:symmetric-square-term-equivalence} holds precisely when $I(c_0,c_1,c_2)=0$.

Finally, the expression for $b$ appearing in the statement of this theorem is $f$ here as we mentioned before, and $r$ is written in terms of $b$ via the first equation in \eqref{eq:system}.\end{proof}

\subsection{Algorithm} \label{subsection:reduce-order}

\subsubsection{The gauge transformation $G$}
\label{prerequisite}
Consider an irreducible third order operator $L_3\in \mathbb{C}(x)[\tau]$.
For any nonzero $G = b_0 + b_1 \tau + b_2 \tau^2 \in \mathbb{C}(x)[\tau]$, $G(V(L_3))$ is the solution space of some third order operator that we will denote as $L_G$. 
Then $L_G G$ is right divisible by $\LCLM(L_3,G)$ and hence equals it (by comparing orders). 
Therefore $L_{G}$ equals $\LCLM(L_3, G)$ right-divided by $G$.

In the algorithm below, we initially do not know the correct values of $b_0, b_1, b_2$ in $\mathbb{C}(x)$.
Therefore in Step~5, $b_0, b_1, b_2$ are first viewed as variables, and $G$ is in $\mathbb{C}(x)[b_0,b_1,b_2][\tau]$. Now $G$ induces a map 
\begin{equation}
    G_2: V(L_3^{\cs 2}) \to V(L_G^{\cs 2}). \label{G2}
\end{equation}
We can write $G_2 = \sum_{i=0}^5 a_i \tau^i$ with $a_i\in \mathbb{C}(x)$. This $G_2$ needs to meet the following requirement: If $u \in V(L_3)$, then $u^2 \in V(L_3^{ \cs 2})$, and with $G$ sending $u$ to $G(u)$, we need $G_2$ to send $u^2$ to $G(u)^2$. The latter is a quadratic polynomial in $b_0, b_1, b_2$ over $R$, where $R$ is defined to be $\mathbb{C}(x)[u^2, u\tau(u), \ldots, \tau^2(u)^2]$.

The inclusion
$$\mathbb{C}(x)[ u^2, \tau(u)^2, \ldots, (\tau^5(u))^2]  \subseteq R$$ is computed in Step~4 of the algorithm below. It allows us to write $G_2( u^2 )$ as an $R$-linear combination of $a_0, \ldots, a_5$. We then equate that to $G(u)^2$, which is is an $R$-linear combination of $b_0^2, b_0b_1,\dots,b_2^2$. Taking coefficients with respect to $u^2,u\tau(u),\dots,\tau^2(u)^2$ gives six $\mathbb{C}(x)$-linear relations between the $a_i$ and $b_0^2,b_0b_1,\dots,b_2^2$.
Solving those writes each $a_i$ as a quadratic polynomial in $b_0,b_1,b_2$ over $\mathbb{C}(x)$. In this way, we obtain an expression for $G_2$ in $\mathbb{C}(x)[b_0,b_1,b_2][\tau]$ that is quadratic in $b_0,b_1,b_2$.

If~(\ref{G2}) is not injective, then its image has dimension $\le 5$ which is the case handled in \Cref{subsection:simpler-case}. That relates $L$ to an operator $L_G$ for which \Cref{thm:symmetric-square-invariant}(\ref{thm:invar2}) applies.
So we need a subspace $V(L_1) \subset V(L_3^{\cs 2})$ in the kernel of $G_2$. This means that the remainder of $G_2$ right-divided by $L_1$ needs to be 0. This $L_1$ will have order~1, and the remainder will be a quadratic polynomial in $b_0,b_1,b_2$ over $\mathbb{C}(x)$. Next we have to find a nontrivial solution for that quadratic polynomial (i.e., find a point on a conic over $\mathbb{C}(x)$). Substituting that solution into $G$ ensures that this remainder is zero, so that~(\ref{G2}) is not injective.

\subsubsection{Algorithm Steps}\hfill

Implementation is in \cite{algo}.

\textbf{Algorithm:} \texttt{ReduceOrder}\\
\textbf{Input:} An absolutely irreducible third order operator $L_3 \in D$.\\
\textbf{Output:} 
$(G,\widetilde{G},L_{2},r)$ if $L_{3}$ is $2$-solvable, where $G$ is a gauge transformation, $\widetilde{G}$ is the inverse of $G$ and sends $L_3$ to an operator of the form~(\ref{eq:symmetric-square-term-equivalence}), and $L_2$ and $r$ are from~\eqref{eq:symmetric-square-term-equivalence}. Otherwise: \texttt{Not\,2-solvable}.
\begin{enumerate}
\item $L_6 \coloneqq  L_3^{\cs 2}$.
\item If $\ord(L_6)=5$, then return $G = 1$, $\widetilde{G}=1$ and return $L_2, r$ from \Cref{thm:symmetric-square-invariant}~\eqref{thm:invar3}~(c).

\item Now $\ord(L_6)=6$. If it exists, let $L_1$ be an order-1 right-factor of $L_6$. Otherwise: return \ \texttt{Not\,2-solvable}.

\item\label{step4} Let $u$ be a symbolic solution of $L_3$, which means that $u(x), u(x+1),u(x+2)$ are symbols, and that we write $u(x+i),i \geq 3$ as $\mathbb{C}(x)$-linear combinations of $u(x), u(x+1),u(x+2)$ in such a way that $L_3(u)=0$.
\item Let $G = b_0 + b_1 \tau + b_2 \tau^2$ and $G_2 = a_0 \tau^0 + \cdots + a_5 \tau^5$ where the $a_i$ and $b_i$ are new variables. Now compute $G_2(u(x)^2)$ and write it as a $\mathbb{C}(x)[a_0,\ldots,a_5]$-linear combination of $u(x)^2, u(x)u(x+1),\ldots,u(x+2)^2$.
\item\label{step6} Compute $G_2(u(x)^2)-(G(u(x)))^2$. Let $S$ be the set of its 6 coefficients as a quadratic polynomial in $u(x),u(x+1),u(x+2)$.
\item Solve $S$ to write $a_0,\ldots,a_5$ as quadratic polynomials in $b_0,b_1,b_2$ over $\mathbb{C}(x)$. Substitute this in $G_2$ so that $G_2 \in \mathbb{C}(x)[b_0,b_1,b_2][\tau]$.
\item  Let $C \in \mathbb{C}(x)[b_0, b_1, b_2]$ be the remainder of $G_2$ right-divided by $L_1$. This $C$ is a {\em conic} (quadratic homogeneous polynomial in 3 variables).
\item Simplify conic $C$ with linear transformations to a conic of the form $c_1 X_1^2 + c_2 X_2^2 + c_3 X_3^2 \in \mathbb{C}(x)[X_1,X_2,X_3]$.
\item Apply a conic-solver \cite{van2006solving} to find a nontrivial point $(X_1,X_2,X_3) \in \mathbb{C}(x)^3$. Reverse the transformations to find a point $(b_0, b_1, b_2) \in \mathbb{C}(x)^3$ on $C$.
\item Substitute the point into $G$ to obtain $G \in \mathbb{C}(x)[\tau]$.
\item Compute $L_G$ by right-dividing ${\rm LCLM}(L_3,G)$ by $G$.
\item As explained in \Cref{prerequisite}, the symmetric square of $L_G$ has order~5 and so we can use \Cref{thm:symmetric-square-invariant} to write it as $L_2^{\cs 2} \cs (\tau - r)$ for some $L_2 \in D$ and $r \in \mathbb{C}(x)$.
\item Compute $\widetilde{G}$, the inverse gauge transformation of $G$, with the extended Euclidean Algorithm (see \Cref{remark:gauge}).
\item Return $G, \widetilde{G}, L_2, r$ and stop.\end{enumerate}

\subsection{Examples}
See \cite{algo} to download these examples and the code.

\subsubsection{Example: OEIS A295371 }

Let \[a(n)= \frac{1}{2n}\sum_{k=0}^{n-1}\binom{n-1}{k}\binom{n+k}{k}\binom{2k}{k}(k+2)(-3)^{(n-1-k)}.\]
Zhi-Wei Sun conjectured that $a(n)$ is a positive odd integer for all $n>0$.
We prove\footnote{A referee pointed out that Lemma 2.3 in \cite{sun2018motzkin} proves a closely related statement where $a(n)$ has $\binom{n+1}{k}$ instead of $\binom{n-1}{k}$.}
this conjecture by using Algorithm \texttt{ReduceOrder} to find a formula for $a(n)$ without $n$ in the denominator. 
OEIS lists the recurrence
\begin{multline*}
    L_3=(2x + 1)(x + 3)^2\tau^3 - (2x + 1)(7x^2 + 38x + 52)\tau^2 \\
    - 3(2x + 5)(7x^2 + 4x + 1)\tau + 27(2x + 5)x^2
\end{multline*}
for $a(n)$. Algorithm \texttt{ReduceOrder} returns the pair $L_{2}, \tau - r$ given by
\begin{equation*}
    \begin{split}
        &\tau^2 + \tau - \frac{3(x+1)}{4(x-1)},\\
        &\tau - \frac{4(x^2 + 3x + 3)(2x + 5)(x-1)^2}{(x^2 + x + 1)(2x + 3)x(x + 1)},
    \end{split}
\end{equation*}
as well as the gauge transformations $G$ and $\widetilde{G}$, omitted here, between the solution spaces of $L_3$ and $L_2^{\cs 2} \cs (\tau-r)$.

Next, an implementation from \cite{cha2010solving} solves $L_2$ in terms of another OEIS entry. After simplifying, we find
\[
    a(n)= (b(n)^2 + 3  b(n-1)^2)/4,
\]
where $b(n)$ is the OEIS entry A002426. To verify that no errors were introduced during the simplification or other steps, we can prove this formula by computing recurrences for $a(n)$ and $b(n)$, and comparing the recurrences and initial values. Analyzing $b(n) \pmod{2}$ proves that $a(n)$ is an odd integer.

Solving equations in closed form, or, as in this paper, solving
equations in terms of solutions of other equations, is often helpful
for proofs. This is not limited to difference equations; one can also prove that $a(n) \in \mathbb{Z}$ by computing closed form solutions of the
differential equation for the generating function.

{\subsubsection{Example: A178808} Let \[a(n) = \frac{1}{n^2}\sum_{k=0}^{n-1}(2k+1)(D_k)^2, \] where $D_0, D_1, \cdots$ are central Delannoy numbers given by A001850.

Zhi-Wei Sun conjectured that $a(n)$ is always an integer and proved it in arXiv:1008.3887.

Our algorithm leads to this formula \[a(n) = \frac{6c(n)c(n-1) - c(n)^2 - c(n-1)^2}{8}\]
where $c(n)$ is the OEIS entry A001850. We can use it to prove the conjecture.
}

{\subsubsection{Example: A268138}
Let \[	a(n) = \frac{1}{n}\sum_{k=0}^{n-1}A001850(k)A001003(k+1).\]
The OEIS tells us that Zhi-Wei Sun conjectured that $a(n)$ is an odd integer, and essentially proved it in arXiv:1602.00574. It can also be proved from the following formula produced with our algorithm $a(n)=$
\[ \frac{(n(n+1)c^2(n)+(n-1)^2c^2(n-1)-3(n-1)(2n+1)c(n)c(n-1)}{2} \]

where $c(n)$ is the OEIS entry A001003. We can use it to prove the conjecture.}

\section{Future Work} \label{futurework}

For order 4 linear difference operators $L$, we want to prove that at least one of the following cases holds: 
\begin{enumerate}
    \item $L$ is reducible in $\mathbb{C}(x)[\tau]$;
    \item $L$ is irreducible but not absolutely irreducible; \\
    (then $L$ is gauge equivalent to an element of $\mathbb{C}(x)[\tau^2]$);
    \item $L\pda^{1}_{2}$ is gauge equivalent to $L_{2a} \cs L_{2b}$ for some second order operators $L_{2a}, L_{2b}\in D$;
    \item $L$ is gauge equivalent to some $L_2^{ \cs 3} \cs (\tau-r)$;
    \item $L$ is not 2-solvable.
\end{enumerate}
If we can prove this, then we also want to develop algorithms for each case. The factorizer for case (1) is already in Maple. Case~(2) is illustrated in \Cref{A260772}. For case~(3), one option is to factor the exterior square of $L\pda^{1}_{2}$, because in case~(3) this should have two factors of order~3, which can then be processed with the algorithm from \Cref{subsection:reduce-order}. However, other options for case~(3) should also be explored to see if they may be more efficient, and an algorithm for case~(4) needs to be developed as well.

\bibliographystyle{ACM-Reference-Format}
\bibliographystyle{plain}
 \newcommand{\noop}[1]{}

\end{document}